\documentclass[11pt]{amsart}
\usepackage{geometry}                
\usepackage{graphicx}
\usepackage{amssymb}
\usepackage{amscd}
\usepackage{epstopdf}
\title{Almost-K\"ahler four manifolds with harmonic self-dual Weyl curvature}

\begin{document}
\author{Inyoung Kim}

\maketitle

 \begin{abstract}
 We show that a compact almost-K\"ahler four manifold $(M, g, \omega)$ with harmonic self-dual Weyl curvature and constant scalar curvature is K\"ahler if $c_{1}\cdot\omega\geq0$.
We also prove an integral curvature inequality for compact almost-K\"ahler four manifolds with harmonic self-dual Weyl curvature. 
\end{abstract}

\maketitle
\section{\large\textbf{Introduction}}\label{S:Intro}
Let $(M, \omega)$ be a symplectic four-manifold. Then $(M, \omega)$ admits a compatible almost-complex structure $J$, 
that is, $\omega(X, Y)=\omega(JX, JY)$ and $\omega(v, Jv)>0$ for any nonzero $v\in TM$ [12]. 
If we define a metric by $g(X, Y):=\omega(X, JY)$, then $(M, g, \omega, J)$ is called an almost-K\"ahler structure. When $J$ is integrable, $\omega$ is parallel and this defines a K\"ahler structure. 

The curvature operator of an oriented four-dimensional riemannian manifold decomposes according to the decomposition of $\Lambda^{2}=\Lambda^{+}\oplus \Lambda^{-}$ as follows. 

\[ 
\mathfrak{R}=
\LARGE
\begin{pmatrix}

 \begin{array}{c|c}
\scriptscriptstyle{W_{+}\hspace{5pt}\scriptstyle{+}\hspace{5pt}\frac{s}{12}}I& \scriptscriptstyle{ric_{0}}\\
 \hline
 
 \hspace{10pt}
 
 \scriptscriptstyle{ric_{0}^{*}}& \scriptscriptstyle{W_{-}\hspace{5pt}\scriptstyle{+}\hspace{5pt}\frac{s}{12}}I\\
 \end{array}
 \end{pmatrix}
 \]

If $W_{-}=0$, then $g$ is called to be self-dual. 
If $ric_{0}=0$, then $g$ is called to be Einstein. 
A compact almost-K\"ahler Einstein manifold with nonnegative scalar curvature is K\"ahler [14]. 
In a four-dimensional case, it was shown in [8] that a compact almost-K\"ahler four manifold with $\delta W_{+}=0$ and nonnegative scalar curvature is K\"ahler with constant scalar curvature.

Using methods in proofs in [8], we prove two results. 
First, we show that a compact almost-K\"ahler four manifold with $\delta W_{+}=0$ and constant
scalar curvature is K\"ahler if $c_{1}\cdot\omega\geq 0$. 
This shows in particular a compact almost-K\"ahler Einstein four manifold with 
$c_{1}\cdot\omega\geq 0$ is K\"ahler-Einstein. 

\vspace{20pt}

\hrule

\vspace{20pt}
$\mathbf{Keywords}$: almost-K\"ahler, four manifold, self-dual Weyl curvature, Einstein metric

$\mathbf{MSC}$: 53C21, 53C24, 53C25. 

Republic of Korea

Email address: kiysd5@gmail.com

This result is not new since there is a more general result in [4]. 
On the other hand, our hypothesis $\delta W_{+}=0$ is weaker than an Einstein condition.

Secondly, we show that for a compact almost-K\"ahler four-manifold with $\delta W_{+}=0$, 
we have 
\[\int_{M}\frac{s^{2}}{24}d\mu\geq \int_{M}|W_{+}|^{2}d\mu.\]
The same result was shown in [10]. 
Our proof entirely depends on the formulas given in [8]. 
Using this inequality, it was shown that a compact almost-K\"ahler four manifold with $J$-invariant Ricci tensor is self-dual K\"ahler-Einstein if the scalar curvature is constant and $b_{-}=0$ [5].
We include a proof in case of compact almost-K\"ahler Einstein four manifolds.
\vspace{20pt}

\section{\large\textbf{Almost-K\"ahler four manifolds with curvature conditions}}\label{S:Intro}
Dr$\breve{a}$ghici showed that a compact almost-K\"ahler four manifold $(M, g, \omega)$ with $J$-invariant Ricci tensor and constant scalar curvature is K\"ahler if $c_{1}\cdot\omega\geq 0$ [4]. 
In this article, we consider $\delta W_{+}=$ condition instead of $J$-invariant Ricci tensor. 
Let $(M, g, \omega)$ be a compact almost-K\"ahler four manifold. Then $\omega$ is a self-dual 
harmonic 2-form with length $||\omega||=\sqrt{2}$
and this gives the following formula.
\[0=\Delta\omega=\nabla^{*}\nabla\omega-2W_{+}(\omega)+\frac{s}{3}\omega.\]
LeBrun showed that this formula with $\delta W_{+}=0$ implies the following integral equality [8].

\newtheorem{Proposition}{Proposition}

\begin{Proposition}
(LeBrun) Let $(M, g, \omega)$ be a compact almost-K\"ahler four-manifold with $\delta W_{+}=0$. 
Then 
\[\int_{M}sW_{+}(\omega, \omega)d\mu=8\int_{M}(|W_{+}|^{2}-\frac{1}{2}|W_{+}(\omega)^{\perp}|^{2})d\mu.\]
\end{Proposition}
The proof of Lemma 2 in [8] contains a stronger result that the statement in Lemma 2. 

\newtheorem{Lemma}{Lemma}
\begin{Lemma}
(LeBrun) Let $(M, g, \omega)$ be an almost-K\"ahler four manifold. Then 
\[|W_{+}|^{2}-|W_{+}(\omega)^{\perp}|^{2}\geq\frac{3}{8}[W_{+}(\omega, \omega)]^{2},\]
with equality at points where $W_{+}(\omega)^{\perp}=0$. 
\end{Lemma}
Lemma 1 implies in particular that 
\[|W_{+}|^{2}-\frac{1}{2}|W_{+}(\omega)^{\perp}|^{2}\geq\frac{3}{8}[W_{+}(\omega, \omega)]^{2}.\]
Using this in Proposition 1, we get the following result [8].

\begin{Proposition}
(LeBrun) Let $(M, g, \omega)$ be a compact almost-K\"ahler four-manifold with $\delta W_{+}=0$. 
Then 
\[\int_{M}W_{+}(\omega, \omega)(W_{+}(\omega, \omega)-\frac{s}{3})d\mu\leq 0.\] 
\end{Proposition}

\newtheorem{Theorem}{Theorem}
\begin{Theorem}
Let $(M, g, \omega)$ be a compact almost-K\"ahler four-manifold with $\delta W_{+}=0$ and constant scalar curvature. 
If $c_{1}\cdot\omega\geq0$, then the scalar curvature is nonnegative and $(M, g, \omega)$ is a K\"ahler surface with constant scalar curvature. 
\end{Theorem}
\begin{proof}
We note that 
\[s^{*}:=2R(\omega, \omega)=2\left(W_{+}(\omega, \omega)+\frac{s}{12}|\omega|^{2}\right).\]
From this, we get 
\[W_{+}(\omega, \omega)=\frac{s^{*}}{2}-\frac{s}{6}.\]
Using this in Proposition 2, we get 
\[\int_{M}(3s^{*}-s)(s^{*}-s)d\mu\leq 0.\]
Suppose $s$ is negative constant. Then from 
\[0\leq 4\pi c_{1}\cdot\omega=\int_{M}\frac{s+s^{*}}{2}d\mu,\]
we have 
\[\int_{M}(3s^{*}-s)(s^{*}-s)d\mu=\int_{M}(3(s^{*})^{2}-4s^{*}s+s^{2})d\mu\]
\[=\int_{M}(-4s(s^{*}+s)+5s^{2}+3(s^{*})^{2})d\mu\geq 0.\]
Thus, we get $s=0$, which is a contradiction. 
Thus, $s\geq 0$. One of the main results in [8] is that a compact almost-K\"ahler four manifold
with $s^{*}\geq 0$ and $\delta W_{+}=0$ is a K\"ahler surface with constant scalar curvature. Using
\[0=\frac{1}{2}\Delta|\omega|^{2}+|\nabla\omega|^{2}-2W_{+}(\omega, \omega)+\frac{s}{3}|\omega|^{2}
=|\nabla\omega|^{2}-2W_{+}(\omega, \omega)+\frac{s}{3}|\omega|^{2},\]
we get 
\[s^{*}=2\left(W_{+}(\omega, \omega)+\frac{s}{12}|\omega|^{2}\right)=s+|\nabla\omega|^{2}.\]
Thus, $s^{*}-s\geq 0$, with equality on $M$ if and only if $(M, g, \omega)$ is K\"ahler. 
Then if $s^{*}\geq 0$, we have 
\[\int_{M}(3s^{*}-s)(s^{*}-s)d\mu=\int_{M}(s^{*}-s)^{2}+2s^{*}(s^{*}-s)d\mu\geq 0.\]
Since $\int_{M}(3s^{*}-s)(s^{*}-s)d\mu\leq 0$ by Proposition 2, we get $s^{*}\equiv s$ if $s^{*}\geq 0$. 
Since $s^{*}\geq 0$ if $s\geq 0$, we get a compact almost-K\"ahler four manifold with $\delta W_{+}=0$
is K\"ahler if $s\geq 0$. 
We note that $s$ on a K\"ahler surface is constant if $\delta W_{+}=0$.

\end{proof}

Since an Einstein metric has $\delta W_{+}=0$ and constant scalar curvature, we get the following corollary. 

\newtheorem{Corollary}{Corollary}
\begin{Corollary}
Let $(M, g, \omega)$ be a compact almost-K\"ahler Einstein four manifold. 
Suppose $c_{1}\cdot\omega\geq 0$. Then $(M, g, \omega)$ is K\"ahler-Einstein. 
\end{Corollary}

This result can be also proved from Dr$\breve{a}$gichi's result [4] 
since an almost-K\"ahler Einstein four manifold has $J$-invariant Ricci tensor and has constant scalar curvature. 

Let $(M, \omega)$ be a symplectic four manifold. If $c_{1}\cdot\omega>0$, then by [15], $b_{+}=1$. 
If a symplectic four manifold has $b_{+}=1$ and $c_{1}\cdot\omega>0$, then $M$ is diffeomorphic to a rational or ruled surface [11]. 
Conversely, there is a following result [9]. 

\begin{Proposition}
(LeBrun) Let $M$ be a smooth four-manifold which is diffeomorphic to a rational or ruled surface. 
Suppose $2\chi+3\tau\geq0$. Then $c_{1}\cdot\omega>0$ for any symplectic form $\omega$ on $M$. 
\end{Proposition}

This gives the following Corollary, which was pointed out in [2]. 

\begin{Corollary}
Let $M$ be a smooth four manifold which is diffeomrophic to a rational or ruled surface. 
Then an almost-K\"ahler Einstien metric $(M, g, \omega)$ is K\"ahler-Einstein with positive scalar curvature.
\end{Corollary}
\begin{proof}
From 
\[2\chi+3\tau=\frac{1}{4\pi^{2}}\int_{M}\left(\frac{s^{2}}{24}+2|W_{+}|^{2}-\frac{|ric_{0}|^{2}}{2}\right)d\mu, \]
we get $2\chi+3\tau\geq 0$ for an Einstein metric. 
By Corollary 1, $(M, g, \omega)$ is K\"ahler-Einstein since $c_{1}\cdot\omega>0$ by Proposition 3. 
For a K\"ahler surface, we have 
\[4\pi c_{1}\cdot[\omega]=\int_{M}s d\mu. \]
From this, we get $s>0$. 
\end{proof}

We prove Proposition 1 in case of compact almost-K\"ahler Einstein four manifolds using the curvature formula of Blair connection given in [7]. 
Let $(M, g, \omega)$ be an almost-K\"ahler four manifold. We consider 
\[\tilde{\nabla}_{X}Y:=\nabla_{X}Y-\frac{1}{2}J(\nabla_{X}J)Y,\]
where $\nabla$ is Levi-Civita connection. Then $\tilde{\nabla}$ induces the connection on the anti-canonical line bundle. 
Then the curvature of $\tilde{\nabla}$ on the anti-canonical line bundle of an almost-K\"ahler Einstien four manifold is given by 
\[iF_{\tilde{\nabla}}=W_{+}(\omega)^{\perp}+\frac{s+s^{*}}{8}\omega+\frac{s-s^{*}}{8}\hat{\omega}.\]
The anti-self-dual 2-form $\hat{\omega}$ is defined on $U$ and $|\hat{\omega}|=\sqrt{2}$, where 
$U=\{p\in M|s^{*}(p)\neq s(p)\}$.

\begin{Corollary}
Let $(M, g, \omega)$ be a compact almost-K\"ahler Einstein four-manifold. Then 
\[\int_{M}\left(\frac{ss^{*}}{8}-\frac{s^{2}}{24}\right)d\mu=\int_{M}(2|W_{+}|^{2}-|W_{+}(\omega)^{\perp}|^{2})d\mu.\]
\end{Corollary}
\begin{proof}
For a compact almost-K\"ahler Einstein four-manifold, we have 
\[c_{1}^{2}=\frac{1}{4\pi^{2}}\int_{M}(|iF_{\nabla}^{+}|^{2}-|iF_{\nabla}^{-}|^{2})d\mu\]
\[=\frac{1}{4\pi^{2}}\int_{M}\left(|W_{+}(\omega)^{\perp}|^{2}+\left(\frac{s+s^{*}}{8}\right)^{2}|\omega|^{2}
-\left(\frac{s-s^{*}}{8}\right)^{2}|\hat{\omega}|^{2}   \right)d\mu.\]

Using 
\[c_{1}^{2}=2\chi+3\tau=\frac{1}{4\pi^{2}}\int_{M}\left(\frac{s^{2}}{24}+2|W_{+}|^{2}\right)d\mu,\]
the result follows.

\end{proof}

\vspace{20pt}

\section{\large\textbf{Integral curvature inequality}}\label{S:Intro}

Using Seiberg-Witten theory, LeBrun showed the following result [6].

\begin{Theorem}
(LeBrun) Let $(M, g)$ be a compact, Einstein four manifold. 
Assume $M$ admits an almost-complex structure $J$ for which the Seiberg-Witten invariant is nonzero.
Suppose $M$ is not finitely covered by $T^{4}$. Then with respect to the orientation by $J$, $\chi\geq 3\tau$. 
Moreover, $\chi=3\tau$ if and only if up to rescaling, the universal cover of $(M, g)$ is complex-hyperbolic 2-space with the standard metric. 
\end{Theorem}

The following result was also shown in [10]. 

\begin{Theorem}
Let $(M, g, \omega)$ be a compact almost-K\"ahler four-manifold with $\delta W_{+}=0$. 
Then $\int_{M}\frac{s^{2}}{24}d\mu\geq \int_{M}|W_{+}|^{2}d\mu$, 
with equality if and only if $(M, g, \omega)$ is K\"ahler with constant scalar curvature. 
\end{Theorem}
\begin{proof}
By Proposition 1, we have 
\[\int_{M}|W_{+}|^{2}d\mu=\int_{M}\left((|W_{+}(\omega)^{\perp}|^{2}-|W_{+}|^{2})+\frac{ss^{*}}{8}d\mu-\frac{s^{2}}{24}\right)d\mu.\]
Using Lemma 1, we get 
\[|W_{+}|^{2}-|W_{+}(\omega)^{\perp}|^{2}\geq \frac{3}{8}|W_{+}(\omega, \omega)|^{2}.\]
Then we get 
\[\int_{M}\frac{s^{2}}{24}d\mu-\int_{M}|W_{+}|^{2}d\mu
\geq\int_{M}\left(\frac{s^{2}}{24}+\frac{3}{8}|W_{+}(\omega, \omega)|^{2}-\frac{ss^{*}}{8}+\frac{s^{2}}{24}\right)d\mu\]
\[=\int_{M}\left(\frac{s^{2}}{12}+\frac{3}{8}\left(\frac{s^{*}}{2}-\frac{s}{6}\right)^{2}-\frac{ss^{*}}{8}\right)d\mu
=\int_{M}\frac{9(s^{*}-s)^{2}}{96}d\mu\geq 0.\]
Moreover, $\int_{M}\frac{s^{2}}{24}d\mu=\int_{M}|W_{+}|^{2}d\mu$ if and only if 
$s^{*}\equiv s$, which implies $(M, g, \omega)$ is K\"ahler. 
\end{proof}

Using this, we get the following result in case of almost-K\"ahler Einstein four manifolds without using Seiberg-Witten theory. 

\begin{Corollary}
Let $(M, g, \omega)$ be a compact almost-K\"ahler Einstein four-manifold. 
Then we have $\chi\geq 3\tau$, with
equality if and only if $(M, g, \omega)$ is self-dual K\"ahler-Einstein. 
\end{Corollary}
\begin{proof}
From the following formulas, 
\[\chi=\frac{1}{8\pi^{2}}\int_{M}\left(\frac{s^{2}}{24}+|W_{+}|^{2}+|W_{-}|^{2}-\frac{|ric_{0}|^{2}}{2}\right)d\mu,\]
\[\tau=\frac{1}{12\pi^{2}}\int_{M}(|W_{+}|^{2}-|W_{-}|^{2})d\mu,\]
we have 
\[\chi-3\tau=\frac{1}{8\pi^{2}}\int_{M}\left(\frac{s^{2}}{24}-|W_{+}|^{2}+3|W_{-}|^{2}-\frac{|ric_{0}|^{2}}{2}\right)d\mu.\]
For an Einstein metric, we have $\delta W_{+}=0$, and therefore by Theorem 3, we get 
\[\chi-3\tau\geq0,\]
with equality if and only if $(M, g, \omega)$ is K\"ahler-Einstein and self-dual.
\end{proof}

\begin{Corollary}
Let $(M, g, \omega)$ be a compact almost-K\"ahler Einstein four manifold. 
Suppose $b_{-}=0$. Then $(M, g, \omega)$ is self-dual K\"ahler-Einstein and $b_{1}=0$. 
\end{Corollary}
\begin{proof}
By Corollary 4, we have $\chi\geq 3\tau$. 
Since $\chi=2-b_{1}+b_{+}-b_{3}$ and $\tau=b_{+}$, 
we have 
\[\chi-3\tau=2-b_{1}-b_{3}-2b_{+}\geq 0.\]
Since $b_{+}\geq 1$, we get $b_{1}=b_{3}=0$ and $b_{+}=1$. 
Moreover, $\chi=3\tau$. 
The result follows from Corollary 4.
\end{proof}

\begin{Corollary}
(Satoh [13]) Let $(M, g, \omega)$ be a compact almost-K\"ahler four manifold with $\delta W_{+}=0$. 
Suppose $\int_{M}s^{2}d\mu=32\pi^{2}(2\chi+3\tau)$.
Then $(M, g, \omega)$ is K\"ahler-Einstein. 
\end{Corollary}
\begin{proof}
From $\int_{M}s^{2}d\mu=32\pi^{2}(2\chi+3\tau)$, we get
\[\int_{M}\left(2\left(\frac{s^{2}}{24}-|W_{+}|^{2}\right)+\frac{|ric_{0}|^{2}}{2}\right)d\mu=0.\]
Since $\int_{M}\frac{s^{2}}{24}d\mu\geq\int_{M}|W_{+}|^{2}d\mu$, we get 
$\int_{M}\frac{s^{2}}{24}d\mu=\int_{M}|W_{+}|^{2}d\mu$ and $ric_{0}=0$. 
Then by Theorem 3, $(M, g, \omega)$ is K\"ahler-Einstein. 
\end{proof}

\begin{Corollary}
Let $(M, g, \omega)$ be a compact simply connected almost-K\"ahler Einstein four manifold with degenerate spectrum of $W_{+}$. 
Then $(M, g, \omega)$ is K\"ahler-Einstein. 
\end{Corollary}
\begin{proof}
By Derdzi$\acute{n}$ski's theorem [3], there exists a K\"ahler metric $h$ which is conformally equivalent to $g$ or $W_{+}=0$. 
Suppose $W_{+}\neq 0$. 
Since $L^{2}$-norm of $W_{+}$ is conformally invariant, by Theorem 3, we have 
\[\int_{M}\frac{s^{2}_{g}}{24}d\mu_{g}\geq \int_{M}|W_{+, g}|^{2}d\mu_{g}=\int_{M}|W_{+, h}|^{2}d\mu_{h}=\int_{M}\frac{s^{2}_{h}}{24}d\mu_{h}\]
On the other hand, we have 
\[2\chi+3\tau=\frac{1}{4\pi^{2}}\int_{M}\left(\frac{s^{2}_{g}}{24}+2|W_{+, g}|^{2}-\frac{|ric_{0, g}|^{2}}{2}\right)d\mu_{g}\]
\[=\frac{1}{4\pi^{2}}\int_{M}\left(\frac{s^{2}_{h}}{24}+2|W_{+, h}|^{2}-\frac{|ric_{0, h}|^{2}}{2}\right)d\mu_{h}.\]
Since $g$ is Einstein, we have 
\[\frac{1}{4\pi^{2}}\int_{M}\left(\frac{s^{2}_{g}}{24}+2|W_{+, g}|^{2}\right)d\mu_{g}
\leq \frac{1}{4\pi^{2}}\int\left(\frac{s^{2}_{h}}{24}+2|W_{+, h}|^{2}\right)d\mu_{h}.\]
Then we get 
\[\int_{M}\frac{s^{2}_{g}}{24}d\mu_{g}= \int_{M}|W_{+, g}|^{2}d\mu_{g}=\int_{M}|W_{+, h}|^{2}d\mu_{h}=\int_{M}\frac{s^{2}_{h}}{24}d\mu_{h}\]
From Theorem 3, we get $(M, g, \omega)$ is K\"ahler-Einstein.

Suppose $W_{+}=0$. 
From $W_{+}(\omega, \omega)=\frac{s^{*}}{2}-\frac{s}{6}$, we get $s^{*}=\frac{s}{3}$ for an anti-self-dual almost-K\"ahler metric. 
Moreover, there is a point where $s^{*}=s$ unless $5\chi+6\tau=0$ [1]. 
Suppose $5\chi+6\tau\neq 0$. Then, $s=0$ since $s$ is constant. 
Then by [14], $(M, g, \omega)$ is K\"ahler-Einstein. Suppose $5\chi+6\tau=0$. Since $2\chi+3\tau\geq 0$ for an Einstien manifolds, we get $\chi\leq 0$. 
Then, from the following formula, 
\[\chi=\frac{1}{8\pi^{2}}\int_{M}\left(\frac{s^{2}}{24}+|W_{+}|^{2}+|W_{-}|^{2}-\frac{|ric_{0}|^{2}}{2}\right)d\mu,\]
we get $s=0$. 

\end{proof}

Let $M$ be a smooth oriented four manifold which is diffeomorphic to a rational or ruled surface. 
Suppose $M$ admits an almost-K\"ahler structure $(g, \omega)$ and $2\chi+3\tau\geq0$. 
Then LeBrun showed the following inequality,
\[\int_{M}|W_{+}|^{2}d\mu\geq\frac{4\pi^{2}}{3}(2\chi+3\tau)(M).\]
For more general results, we refer to [9]. 
If $M$ admits an almost-K\"ahler Einstein metric, then from this inequality, we get 
 $\int_{M}\frac{s^{2}}{24}d\mu\leq \int_{M}|W_{+}|^{2}d\mu$. 
On the other hand, from Theorem 3, we have $\int_{M}\frac{s^{2}}{24}d\mu\geq \int_{M}|W_{+}|^{2}d\mu$.
Thus, we get the equality and therefore, $(M, g, \omega)$ is K\"ahler-Einstein positive scalar curvature. 
This gives another proof for Corollary 2.

\vspace{20pt}
$\mathbf{Acknowledgments}$: The author is thankful to Prof. Claude LeBrun for helpful comments.

\end{document}